\documentclass[12pt,reqno]{amsart}
\usepackage{amsfonts}
\usepackage{amsmath}
\usepackage{amssymb,latexsym}
\usepackage[matrix,arrow]{xy}
\usepackage[usenames]{color}
\usepackage{slashed}
\usepackage{parskip}
\usepackage{amsmath,amssymb,graphics}
\usepackage[colorlinks=true,citecolor=blue]{hyperref}
\usepackage{graphicx}

\newcommand{\mathsym}[1]{{}}

\newtheorem{thm}{Theorem}[section]
\newtheorem{cor}[thm]{Corollary}
\newtheorem{lem}[thm]{Lemma}
\newtheorem{prop}[thm]{Proposition}

\theoremstyle{definition}
\newtheorem{defn}{Definition}[section]
\numberwithin{equation}{section}
\theoremstyle{remark}
\newtheorem{rem}{Remark}[section]
\theoremstyle{example}

\newcommand{\n}{\nabla}

\newcommand{\dmu }{\mathrm{d}\mu }
\newcommand{\md }{\mathrm{d}}

\newcommand{\be}{\begin{equation}}
\newcommand{\ee}{\end{equation}}
\newcommand{\ba}{\begin{eqnarray}}
\newcommand{\ea}{\end{eqnarray}}
\newcommand{\ban}{\begin{eqnarray*}}
\newcommand{\ean}{\end{eqnarray*}}
\newcommand{\lf}{\left}
\newcommand{\rt}{\right}

\newcommand{\bpf}{\begin{proof} }
\newcommand{\epf}{\end{proof} }

\begin{document}
\title{An upper bound of the heat kernel along the harmonic-Ricci flow}
\author{Shouwen Fang and Tao Zheng}

\maketitle
\begin{abstract}
In this paper, we first derive a Sobolev inequality along the harmonic-Ricci flow.
We then prove a linear parabolic estimate based on the Sobolev inequality and Moser's iteration.
As an application, we will obtain an upper bound estimate for the heat kernel under the flow.
\end{abstract}
\section{Introduction}
\setcounter{equation}{0}
Let $M^n$ be an $n$ dimensional closed smooth manifold and assume $n\ge 3$.
In \cite{MR}, M\"{u}ller studied a system of the Ricci flow coupled with a harmonic map heat flow
\begin{equation}\left
\{\begin{array}{l}
\large{\partial_t g=-2\text{Ric}+2\alpha(t) \nabla\phi\otimes \nabla\phi},\\
\large{\partial_t \phi=\tau_{g}\phi,}
\end{array}\right.\label{O1}
\end{equation}
where $\phi(\cdot,t):(M,g(\cdot,t))\rightarrow(N,h)$ is a family of smooth maps between two Riemannian manifolds, both $g(\cdot,t)$ and $h$ are Riemannian metrics, $\alpha(t)$ is a positive non-increasing function, and $\tau_{g}\phi$ denotes the intrinsic Laplacian of $\phi$. This flow is also called
harmonic-Ricci flow (cf. \cite{MH,MR,Zh}). The harmonic-Ricci flow may
be one of helpful tools in finding the harmonic map between two Riemannian manifolds.
If the target manifold $N$ is $\mathbb{R}$, the harmonic-Ricci flow reduces to the extended Ricci flow, which was first introduced by List in \cite{BL}. The extended Ricci flow is very useful in general relativity.
If $\phi$ is a constant map, the system (\ref{O1}) degenerates to Hamilton's Ricci flow discussed widely recently,
see for example the book \cite{CLN} and seminal papers \cite{CZ, H1, H2, P1}.
Similarly as Ricci flow and the extended Ricci flow, corresponding theories for the harmonic-Ricci flow were established in \cite{MR}, such as the
short time existence, the $\mathcal{W}$ entropy, the $\mathcal{F}$ entropy, reduced length and reduced volume. Hence the harmonic-Ricci flow may be
investigated through the methods used in the Ricci flow.

In this paper, along the harmonic-Ricci flow, we consider the heat kernel $G(x,t,y,s)$ which is the fundamental solution of the following heat equation
\begin{equation}\label{O2}
(\Delta-\partial_t)u(x,t)=0.
\end{equation}
The estimate for heat kernel has always been an interesting topic in the study of differential equations on manifolds. In their celebrated paper \cite{LY}, Li and Yau derived some point-wise gradient estimates for the positive solutions of (\ref{O2}) on complete manifolds with
fixed metric and lower bounded Ricci curvature, from which the upper and lower bounds on the heat kernel were obtained.
Wang \cite{W} proved a global gradient estimate when the boundary of manifold is nonconvex, and got both upper and lower bounds for the heat kernel with Neumann conditions. Later, in \cite{M, MH, LW} the evolved metrics were studied, and some bounds on heat kernel under some geometric flows (e.g. the Ricci flow and the extended Ricci flow) were also derived with the assistance of the Sobolev inequality.

It is well-known that the Sobolev inequality is an important analytical tool in geometric analysis. Recently, there occur many
interesting results on the Sobolev inequality under different geometric flows, especially the Ricci flow.
In \cite{Hs, KZ, Y1, Y2, Z1, Z2, Z3}, some uniform Sobolev inequalities were proved along Ricci flow by using the monotonicity of Perelman's $\mathcal{W}$ entropy. As a consequence, Perelman's short time non-collapsing was extended to a long time version.
In particular, by the Sobolev inequality, Zhang \cite{Z1} proved a global upper bound for the fundamental solution of a heat equation under backward Ricci flow with the assumption that Ricci curvature is nonnegative and the injectivity radius is bounded from below.

The main purpose of this paper is to establish the uniform Sobolev inequality and an upper bound for the heat kernel under the harmonic-Ricci flow.
For convenience, we denote as in \cite{F1,F2,BL,MR} the symmetric two-tensor field $S_{y}$ with components $S_{ij}$ and its trace $S:=g^{ij}S_{ij}$  by
\begin{equation*}
S_{ij}:=R_{ij}-\alpha(t)\nabla_i \phi\nabla_j \phi\text{\quad and\quad}S:=R-\alpha(t)|\nabla\phi|^2,
\end{equation*}
where $R_{ij}$ and $R$ are the Ricci curvature components and the scalar curvature of $(M, g)$ respectively.
Using the monotonicity of the $\mathcal{W}$ entropy, we obtain the following Sobolev inequality.
\begin{thm}\label{th1}
Let ($g(x,t), \phi(x,t)$) be a solution to the harmonic-Ricci flow (\ref{O1}) in $M^n\times[0,T)$ with initial value ($g_0, \phi_0$).
Let  $A_0$ and $B_0$ be positive numbers such that the following $L^2$ Sobolev inequality holds initially,
i.e. for any $v\in W^{1,2}(M, g_0)$,
$$
\lf(\int_M v^{\frac{2n}{n-2}}\md\mu\lf(g_0\rt)\rt)^{\frac{n-2}{n}}\leq A_0 \int_M\lf(|\nabla v|^2+\frac{1}{4}Sv^2\rt)\md\mu\big(g_0\big)+B_0 \int_M v^2\md\mu\lf(g_0\rt).
$$
Then for all $v\in W^{1,2}(M,g(t))$ we have
\begin{eqnarray*}
\lf(\int_M v^{\frac{2n}{n-2}}\md\mu\big(g(t)\big)\rt)^{\frac{n-2}{n}}\leq A(t) \int_M \lf(|\nabla v|^2+\frac{1}{4}Sv^2\rt)\md\mu\big(g(t)\big)
+B(t)\int_M v^2\md\mu\big(g(t)\big),
\end{eqnarray*}
where $A(t)=Ce^{\frac{8tB_0}{n A_0}}A_0$, $B(t)=Ce^{\frac{8tB_0}{n A_0}}B_0$, and $C$ is a positive constant depending only on $A_0$, $B_0$, $g_0$,
$\phi_0$ and $n$. In particular, if $S\geq 0$ at the initial time, then $C$ is a positive constant depending only on $n$.
\end{thm}
Based on the Sobolev inequality and Moser's iteration, Ye \cite{Y1} proved a linear parabolic estimate under the Ricci flow, which was applied to get the upper bound of curvature tensor. Jiang \cite{J} gave a linear parabolic estimate along the K\"ahler-Ricci flow, from which he obtained upper bound estimates of the scalar curvature and the gradient of Ricci potential. Here from the above Sobolev inequality, we can get the following linear parabolic estimate.
\begin{thm}\label{th2}
Assume that ($g(x,t), \phi(x,t)$) is a smooth solution to the harmonic-Ricci flow (\ref{O1}) in $M^n\times[0,T]$ with initial value ($g_0, \phi_0$) and
$S\geq0$ at the initial time.
Let $f$ be a nonnegative Lipschitz continuous function on $M\times[0, T]$ satisfying
\begin{equation}\label{O3}
\partial_t f\leq \Delta f+af
\end{equation}
on $M\times[0, T]$ in the weak sense, where $a\geq 0$. Then we have for any $0<t\leq T$ and $p>0$
\begin{equation*}
\sup_{x\in M} |f(x,t)|\leq \left(C_1a+\frac{C_2}{t}\right)^{\frac{n+2}{2p}}\left(\int_{0}^T\int_Mf^{p}\md\mu \md t\right)^{\frac{1}{p}},
\end{equation*}
where $C_1$ and $C_2$ are both positive constants depending on dimension $n$, $p$, $g_0$, $\phi_0$ and the first eigenvalue $\lambda_0$ of $\mathcal{F}$ entropy with respect to $g_0$.
\end{thm}
Obviously, the heat equation (\ref{O2}) is a simple linear parabolic equation and the heat kernel satisfies naturally
the conditions of the above theorem. As a consequence of Theorem \ref{th2}, it is not difficult to get an upper bound of the heat kernel under the
harmonic-Ricci flow, which is similar to the upper bound in Bailesteanu \cite{M}, Bailesteanu and Tran \cite{MH}, and Wang \cite{W}.
But our upper bound depends on the first eigenvalue of $\mathcal{F}$ entropy, which is different from their results.
More precisely,  we prove
\begin{thm}\label{th3}
Assume that ($g(x,t), \phi(x,t)$) is a smooth solution to the harmonic-Ricci flow (\ref{O1}) in $M^n\times[0,T]$ with initial value ($g_0,\phi_0$)
and $S\geq0$ at the initial time. Let $G(x,t;y,s)$ be the heat kernel.
Then there exists a positive constant $C$, which depends on dimension $n$, $g_0$, $\phi_0$ and the first eigenvalue $\lambda_0$ of $\mathcal{F}$ entropy with
respect to $g_0$, such that
\begin{equation*}
G(x,t;y,s)\leq \frac{C}{(t-s)^{\frac{n}{2}}},
\end{equation*}
for $\forall 0\leq s<t\leq T$, and $\forall x,y\in M$.
\end{thm}
\begin{rem} Here the nonnegativity of $S$ in our theorem is a little weaker than the positivity in Bailesteanu and Tran \cite{MH}.
Indeed, the assumption can be taken away if we impose other reliance of constant $C$ on the upper bound of time (see Corollary \ref{co4}).
\end{rem}
The rest of the paper is organized as follows. In section 2 we consider the $\mathcal{W}$ entropy under the harmonic-Ricci flow and derive the Sobolev inequality by using the monotonicity of $\mathcal{W}$ entropy. In section 3 we show the linear parabolic estimate for (\ref{O3}) and the upper bound estimate of the heat kernel along the harmonic-Ricci flow.

\section{ Sobolev inequalities under the harmonic-Ricci flow}
\setcounter{equation}{0}
In this section, we show a uniform log-Sobolev inequality along the harmonic-Ricci flow from the monotonicity of $\mathcal{W}$ entropy, and
then verify the equivalence of our Sobolev inequality and the uniform log-Sobolev inequality, from which we prove Theorem \ref{th1}.
As a corollary, we obtain the other uniform Sobolev inequality depending on the first eigenvalue of $\mathcal{F}$ entropy,
which will be used in the next section.

First let us introduce the definition of $\mathcal{W}$ entropy via corresponding conjugate heat equation just as Perelman \cite{P1} has done
in Ricci flow. Let $u(x,t)$ be a positive solution to the following conjugate heat equation
\begin{equation}\label{21}
\Delta u-Su+\partial_t u=0.
\end{equation}
From the conjugate heat equation (\ref{21}) and the harmonic-Ricci flow (\ref{O1}), we can get easily
$$
\frac{\md}{\md t}\int_Mu(x,t)\md\mu(g(t))=\int_M(\partial_t-S)u\md\mu(g(t))=-\int_M\Delta u\md\mu(g(t))=0,
$$
where we used the fact that $M$ is closed. Therefore, without loss of generality we assume that $u(x,t)$ satisfies
\begin{equation*}
\int_Mu(x,t)\md\mu(g(t))=1
\end{equation*}
for any $t\in[0,T]$. The $\mathcal{W}$ entropy is given by a functional of the positive solution $u$ of (\ref{21}) as follows.
\begin{defn}
The $\mathcal{W}$ entropy is defined as the following functional
\begin{equation*}
\mathcal{W}(g,f,\tau)=\int_M\big(\tau(S+|\nabla f|^2)+f-n\big)u\md\mu(g(t)),
\end{equation*}
where $f=-\ln u-\frac{n}{2}(\ln 4\pi\tau)$ and $\tau$ is a scaling factor satisfied $\frac{\md\tau}{\md t}=-1$.
\end{defn}

\begin{rem} The same definition can also be found in \cite{BL,MR}.
From the relationship between $f$ and $u$, $\mathcal{W}$ entropy can also be rewritten by the function $u$ directly as follow.
\begin{equation}
\mathcal{W}(g,u,\tau):=\int_M\left[\tau\lf(Su+\frac{|\nabla u|^2}{u}\rt)-u\ln u-\frac{n}{2}\ln( 4\pi\tau)u -nu \right]\md\mu(g(t)).\label{22}
\end{equation}
\end{rem}
Now let us recall the following monotonicity formula, which had been proved in Theorem 5.2 of \cite{F4} for general geometric flow and Proposition 7.1 of \cite{MR}(or Theorem 6.1 of \cite{BL}) for the case of constant $\alpha$. We omit the details here.
\begin{prop}\label{p1}
 Let ($g(x,t), \phi(x,t)$) be a solution of the harmonic-Ricci flow (\ref{O1}) and $u(x,t)$
be a positive solution of (\ref{21}). Then $\mathcal{W}$ entropy is non-decreasing in $t$. More precisely,
\begin{equation*}
\frac{\md}{\md t}\mathcal{W}=\int_M\left(2\tau|S_y+Hess(f)-\frac{g}{2\tau}|^2+2\tau\alpha\left\lvert \tau_g \phi-\langle\nabla\phi, \nabla f\rangle\right\lvert^2-\tau\dot{\alpha}|\nabla\phi|^2\right)u\md\mu(g(t))
\ge0.
\end{equation*}
\end{prop}

To prove Theorem \ref{th1}, we first need to prove the corresponding log-Sobolev inequality for any $t\in[0,T)$. Here we use the same method as Zhang
\cite{Z2} and Liu-Wang \cite{LW}. Using the monotonicity of $\mathcal{W}$ entropy, we have the following log-Sobolev inequality.
\begin{lem}[Log-Sobolev Inequality]\label{lm1}
Under the same assumptions of Theorem \ref{th1}. Then for any $t\in[0,T)$, $v\in W^{1,2}(M,g(t))$ with $\int_M v^2\md\mu(g(t))=1$ and any $\epsilon>0$,
we have
\begin{align*}
\int_M v^2\ln v^2\md\mu(g(t))\leq\epsilon^2\int_M\big(4|\nabla v|^2+Sv^2\big)\md\mu(g(t))-n\ln(2\epsilon)
+4(t+\epsilon^2)\frac{B_0}{A_0}+\frac{n}{2}\ln\frac{nA_0}{2e}.
\end{align*}
\end{lem}
\begin{proof}
For any fixed $t_0\in[0,T)$ and any $\epsilon>0$, we set $\tau(t)=\epsilon^2+t_0-t.$
From the monotonicity of the $\mathcal{W}$ entropy in Proposition \ref{p1}, we get
\begin{align}
\nonumber\inf_{\int_M u_0\md\mu(g_0)=1}\mathcal{W}(g_0,f_0,t_0+\epsilon^2)\leq &\mathcal{W}(g_0,\widetilde{f}(\cdot,0),t_0+\epsilon^2)\\
\leq&\mathcal{W}(g(t_0),\widetilde{f}(\cdot,t_0),\epsilon^2)\nonumber\\
=&\inf_{\int_M u \md\mu(g(t_0))=1}\mathcal{W}(g(t_0),f,\epsilon^2),
\label{23}
\end{align}
where $(4\pi\tau)^{-\frac{n}{2}}e^{-\widetilde{f}(\cdot,t)}$ satisfies the conjugate heat equation (\ref{21}), $f_0$ and $f$ are given by the formulas
$u_0=\big(4\pi(t_0+\epsilon^2)\big)^{-\frac{n}{2}}e^{-f_0}$ and $u=(4\pi\epsilon^2)^{-\frac{n}{2}}e^{-f}$.
The last equality holds because the infimum of $\mathcal{W}$ entropy is achieved by a minimizer $\widetilde{f}(\cdot,t_0)$
(cf. Corollary 1.5.9 in \cite{CZ} or section 3 in \cite{P1}). Using (\ref{22}) we  rewrite (\ref{23}) as
\begin{align*}
\inf_{\int u_0\md\mu(g_0)=1}\int_M\left[(\epsilon^2+t_0)(S+|\nabla
\ln u_0|^2)-\ln u_0
-\frac{n}{2}\ln\lf(4\pi(t_0+\epsilon^2)\rt)\right]u_0\md\mu(g_0)\\
\leq\inf_{\int u\md\mu(g(t_0))=1}\int_M\left[\epsilon^2\lf(S+|\nabla \ln u|^2\rt)-\ln u-\frac{n}{2}\ln\lf(4\pi\epsilon^2\rt)\right]u\md\mu(g(t_0)).
\end{align*}
Let $v=\sqrt{u}$ and $v_0=\sqrt{u_0}$, the above inequality leads to
\begin{align}
\inf_{\int v_0^2\md\mu(g_0)=1}\int_M\left[(\epsilon^2+t_0)(Sv_0^2+4|\nabla
v_0|^2)-v_0^2\ln v_0^2\right]\md\mu(g_0)-\frac{n}{2}\ln(t_0+\epsilon^2)\nonumber\\
\leq\inf_{\int v^2\md\mu(g(t_0))=1}\int_M\big[\epsilon^2(Sv^2+4|\nabla v|^2)-v^2\ln v^2\big]\md\mu(g(t_0))-\frac{n}{2}\ln\epsilon^2.
\label{24}
\end{align}
Notice that $\ln x$ is a concave function and $\int v_0^2\md\mu(g_0)=1$, thus applying Jensen's inequality we deduce
\[\int_Mv_0^2\ln v_0^{q-2}\md\mu(g_0)\leq\ln \int v_0^{q-2}v_0^2\md\mu(g_0),\]
where $q=\frac{2n}{n-2}$. This means
$$
\int_Mv_0^2\ln v_0^2\md\mu(g_0)\leq\frac{n}{2}\ln\|v_0\|_q^2,
$$
By the assumption that the Sobolev inequality holds for the initial time $t=0$, combining with the above inequality we have
\[\int_Mv_0^2\ln v_0^2\md\mu(g_0)\leq\frac{n}{2}\ln\left(A_0\int_M\lf(|\nabla v_0|^2+\frac{1}{4}Sv_0^2\rt)\md\mu(g_0)+B_0\right).\]
Moreover, the inequality $\ln z\leq yz-\ln y-1$ holds for any $y,z>0$. Using it in the RHS of the above we arrive at
\begin{equation*}
\int_Mv_0^2\ln v_0^2\md\mu(g_0)\leq\frac{n}{2}y\lf(A_0\int_M\lf(|\nabla v_0|^2+\frac{1}{4}Sv_0^2\rt)\md\mu(g_0)+B_0\rt)-\frac{n}{2}\ln y-\frac{n}{2}.
\end{equation*}
Now we choose $y=\frac{8(t_0+\epsilon^2)}{n A_0}$,
then the above inequality implies
\begin{align}\label{25}
\int_Mv_0^2\ln v_0^2\md\mu(g_0)\leq&(t_0+\epsilon^2)\int_M\lf(4|\nabla v_0|^2+Sv_0^2\rt)\md\mu(g_0)
+\frac{4(t_0+\epsilon^2)B_0}{A_0}\nonumber\\
&-\frac{n}{2}\ln\frac{8(t_0+\epsilon^2)}{nA_0}-\frac{n}{2}.
\end{align}
Substituting (\ref{25}) into (\ref{24}), we conclude that
\begin{align*}
\int_M v^2\ln v^2\md\mu(g(t_0))\leq&\epsilon^2\int_M\lf(4|\nabla v|^2+Sv^2\rt)\md\mu(g(t_0))-n\ln(2\epsilon)\\
&+\frac{4(t_0+\epsilon^2)B_0}{A_0}+\frac{n}{2}\ln\frac{nA_0}{2e}.
\end{align*}
The time $t_0$ is arbitrary, thus the proof of the lemma is completed now.
\end{proof}
In general, the log-Sobolev inequality and the Sobolev inequality are equivalent, which can be proved via the upper bound of heat kernel. More
details can be found in Zhang's Theorem 4.2.1 of \cite{Z3}. But the Sobolev inequality along a geometric flow is different from the general Sobolev
inequality in closed manifolds. So it is necessary to provide the equivalence between our log-Sobolev inequality and Sobolev
inequality here. We can give a proof of the following equivalence lemma by the trick in \cite{BCLS}.
\begin{lem}\label{lm2}
Let $(M^n, g)$ be a  closed Riemannian manifold ($n\geq3$). Then the following inequalities are equivalent up to constants.\\
(I) Sobolev inequality: there exist positive constants $A$ and $B$ such that, for all $v\in W^{1,2}(M)$,
\begin{equation*}
\lf(\int_M v^{\frac{2n}{n-2}}\md\mu\rt)^{\frac{n-2}{n}}\leq A\int_M \lf(|\nabla v|^2+\frac{1}{4}Sv^2\rt)\md\mu+B\int_Mv^2\md\mu;
\end{equation*}
(II) Log-Sobolev inequality: for all $v\in W^{1,2}(M)$ such that $\|v\|_2=1$ and all $\epsilon >0$,
\begin{equation*}
\int_Mv^2\ln v^2 \md\mu\leq \epsilon^2\int_M\lf(|\nabla v|^2+\frac{1}{4}Sv^2\rt)\md\mu-\frac{n}{2}\ln\epsilon^2+BA^{-1}\epsilon^2+\frac{n}{2}\ln\frac{nA}{2e}.
\end{equation*}
\end{lem}
\begin{proof}$I \Rightarrow II:$ The proof is a standard application of the Jensen's inequality. The derivation is almost same with (\ref{25}).
We only need to take $y=\frac{2\epsilon^2}{n A}$ instead, the log-Sobolev inequality will be obtained as desired.

$II \Rightarrow I:$ Notice that the LHS of log-Sobolev inequality is bounded from below for all $v\in W^{1,2}(M)$. Hence,
the log-Sobolev inequality implies directly that for all $v\in W^{1,2}(M)$
$$
A\int_M\lf(|\nabla v|^2+\frac{1}{4}Sv^2\rt)\md\mu+B>0.
$$
 Since the log-Sobolev inequality holds for all $\epsilon>0$, the RHS of log-Sobolev inequality can be seen as a function of
$\epsilon$ and reaches its  minimum. Thus we have
\begin{equation}\label{26}
\int_Mv^2\ln v^2 \md\mu\leq \frac{n}{2}\ln \lf[A\int_M\lf(|\nabla v|^2+\frac{1}{4}Sv^2\rt)\md\mu+B\rt],
\end{equation}
for all $v\in W^{1,2}(M)$ such that $\|v\|_2=1$. Now we consider any function $f\in W^{1,2}(M)$.
By the Kato's inequality $|\nabla|f||\leq|\nabla f|$,
we only need to prove the Sobolev inequality for all nonnegative functions. So we assume that $f\geq0$.
For the sake of conveniences, we denote
$$
W(f)=\lf[A\int_M\lf(|\nabla f|^2+\frac{1}{4}Sf^2\rt)\md\mu+B\int_Mf^2\md\mu\rt]^{\frac{1}{2}}.
$$
Taking $v=\frac{f}{\|f\|_2}$ in (\ref{26}) yields
\begin{equation*}
\int_Mf^2\ln \lf(\frac{f}{\|f\|_2}\rt)^2 \md\mu\leq n\|f\|_2^2\ln \lf(\frac{W(f)}{\|f\|_2}\rt).
\end{equation*}
It is just $(LS_2^q)$ in Page 1067 of \cite{BCLS}, where $q=\frac{2n}{n-2}$. Using their method to treat the above estimate, we arrive at
\begin{equation}\label{27}
\|f\|_2 \leq W(f)^{\theta} \|f\|_s^{1-\theta},
\end{equation}
where $0<s<2$ and $\frac{1}{2}=\frac{\theta}{q}+\frac{1-\theta}{s}$.
Next we need to define a family of functions $f_k$ for $k\in\mathbb{Z}$ by
$$
f_k=\min\{(f-2^k)^+, 2^k\},
$$
where $(f-2^k)^+=\max\{f-2^k,0\}$. From the definition of $f_k$ it is obvious to have the following estimate for any $p>0$
\begin{equation}\label{28}
2^{pk}\mu\{f\geq2^{k+1}\}\leq\int_Mf_k^p\md\mu\leq2^{pk}\mu\{f\geq2^{k}\}.
\end{equation}
Set $a_k=2^{qk}\mu\{f\geq2^{k}\}$. Combining (\ref{27}) with (\ref{28}), we derive
\begin{align*}
a_{k+1}\leq& 2^{q(k+1)-2k}\int_Mf_k^2\md\mu\\
\leq& 2^{q(k+1)-2k}W(f_k)^{2\theta}\|f_k\|_s^{2(1-\theta)}\\
\leq& 2^{q}W(f_{k})^{2\theta }a_{k}^{\frac{2(1-\theta)}{s}}.
\end{align*}
Consequently, by the H\"older inequality we have
\begin{align*}
\sum\limits_{k\in\mathbb{Z}}a_{k}
=&\sum\limits_{k\in\mathbb{Z}}a_{k+1}\\
\leq&\sum\limits_{k\in\mathbb{Z}}2^{q}W(f_{k})^{2\theta }a_{k}^{\frac{2(1-\theta)}{s}}\\
\leq&2^{q}\lf(\sum\limits_{k\in\mathbb{Z}}W(f_{k})^{2}\rt)^{\theta }
\lf(\sum\limits_{k\in\mathbb{Z}}a_{k}^{\frac{2 }{s}}\rt)^{1-\theta }\\
\leq&2^{q}
\lf(\sum\limits_{k\in\mathbb{Z}}W(f_{k})^{2}\rt)^{\theta }
\lf(\sum\limits_{k\in\mathbb{Z}}a_{k}\rt)^{\frac{2(1-\theta)}{s}},
\end{align*}
where the last inequality follows from $0<s<2$. This leads to
\begin{equation}\label{29}
\sum\limits_{k\in\mathbb{Z}}a_{k}\leq 2^{\frac{q(q-s)}{2-s}}\lf(\sum\limits_{k\in\mathbb{Z}}W(f_{k})^{2}\rt)^{\frac{q}{2}}.
\end{equation}
Moreover, it follows from the definition of $a_k$ that
\begin{align}\label{210}
\int_M f^{q}\md\mu
=&\sum\limits_{k\in\mathbb{Z}}\int_{2^{k}\leq f\leq 2^{k+1}}f^q\md\mu\nonumber\\
\leq&\sum\limits_{k\in\mathbb{Z}}2^{q(k+1)}\bigg(\mu(f\geq 2^{k})-\mu(f\geq 2^{k+1})\bigg)\nonumber\\
=&(2^{q}-1)\sum\limits_{k\in\mathbb{Z}}a_{k}.
\end{align}
Now the rest of proof is only need to control the term $W(f_k)$ in (\ref{29}). We have the following key estimate.

{\sc Claim}. If $\frac{A}{4}S+B\geq0$, then for any $0\leq f\in W^{1,2}(M)$ we have
$$
\sum\limits_{k\in\mathbb{Z}}W(f_{k})^{2}
\leq W(f)^2.
$$
Firstly, we note that
\begin{align*}
\int_{M}|f_{k}|^{2}\dmu
=&2\int_{0}^{2^{k+1}-2^{k}} t^{p-1}\mu(f-2^{k}\geq t)\md t\nonumber\\
=&2\int_{2^{k}}^{2^{k+1}}(s-2^{k})^{p-1}\mu(f\geq s)\md s.
\end{align*}
Thus we have
\begin{align*}
\sum\limits_{k\in\mathbb{Z}}\int_{M}|f_{k}|^{2}\dmu
=&2\sum\limits_{k\in\mathbb{Z}}\int_{2^{k}}^{2^{k+1}}(s-2^{k})\mu(f\geq s)\md s\\
=&\sup\limits_{k\in\mathbb{Z}}\lf\{\sup\limits_{s\in[2^{k},\,2^{k+1}]}\frac{ s-2^{k} }{s}\rt\}
\lf(2\sum\limits_{k\in\mathbb{Z}}\int_{2^{k}}^{2^{k+1}}s\mu(f\geq s)\md s\rt)\\
\leq &\frac{ 1 }{2}\int_{M} f ^2\dmu.
\end{align*}
Because of the assumption that
$\frac{A}{4}S+B\geq0$, we can denote $(\frac{A}{4}S+B)\dmu$ as a new measure. By the same method as above we can derive
\begin{equation}\label{211}
\sum\limits_{k\in\mathbb{Z}}\int_{M}\lf(\frac{A}{4}S+B\rt) f_{k} ^{2}\dmu \leq  \frac{1 }{2}\int_{M}\lf(\frac{A}{4}S+B\rt) f ^2\dmu.
\end{equation}
In addition, it also holds that
\begin{align}\label{212}
\sum\limits_{k\in\mathbb{Z}}\int_{M}|\n f_{k}|^{2}\dmu
=&\sum\limits_{k\in\mathbb{Z}}\int_{2^{k}\leq  f \leq 2^{k+1}}|\n f  |^{2}\dmu\nonumber\\
=& \int_{M}|\n f  |^{2}\dmu.
\end{align}
Combining (\ref{211}) with (\ref{212}) yields
\begin{align*}
\sum\limits_{k\in\mathbb{Z}}W(f_{k})^{2}
=&\sum\limits_{k\in\mathbb{Z}}\lf(A\int_{M}\lf(|\n f_{k}|^2+\frac{1}{4}S f_{k} ^{2}\rt)\md \mu+B \int_{M} f_{k} ^{2}\md \mu\rt)\\
\leq& A\sum\limits_{k\in\mathbb{Z}}\int_{M} |\n f_{k}|^2\md \mu
+ \sum\limits_{k\in\mathbb{Z}}\int_{M} \lf(\frac{A}{4}S+B\rt)f_{k} ^{2}\md \mu\\
\leq & A\int_{M} |\n f |^2\md \mu+ \frac{1}{2}\int_{M}\lf(\frac{A}{4}S +B\rt)f^{2}\md \mu\\
\leq &W(f)^2.
\end{align*}
Hence the proof of the claim has been completed. Now we can use the claim to finish the proof of lemma. One should be careful because the assumption in the claim does not have to be true.
But it has no effect on the final proof. Since $M$ is a closed manifold, there exists a nonnegative constant $S_0$ such that $S+S_0\geq0.$
By combining (\ref{29}) and (\ref{210}), we deduce
 \begin{align}\label{213}
\int_M f^{q}\md\mu
\leq &(2^{q}-1)2^{\frac{q(q-s)}{2-s}}\lf(\sum\limits_{k\in\mathbb{Z}}W(f_{k})^{2}\rt)^{\frac{q}{2}}\nonumber\\
\leq &(2^{q}-1)2^{\frac{q(q-s)}{2-s}}\lf(\sum\limits_{k\in\mathbb{Z}}\lf(A\int_{M}\lf(|\n f_{k}|^2+\frac{S+S_0}{4} f_{k} ^{2}\rt)\md \mu+B \int_{M} f_{k} ^{2}\md \mu\rt)\rt)^{\frac{q}{2}}\nonumber\\
\leq &(2^{q}-1)2^{\frac{q(q-s)}{2-s}}\lf( W(f)^2+A\int_{M}\frac{S_0}{4} f^{2}\md \mu\rt)^{\frac{q}{2}},
\end{align}
where the last inequality holds due to the claim. Note that the LHS of (\ref{26}) is bounded from below for all $v\in W^{1,2}(M)$, it implies that
$$
\lambda_1=\inf\lf\{W(f)^2|\int_M f^2\md \mu=1, f\in W^{1,2}(M)\rt\}>0,
$$
i.e.
\begin{equation}\label{214}
\int_{M} f^{2}\md \mu\leq \lambda_1^{-1}W(f)^2,
\end{equation}
for any $f\in W^{1,2}(M)$.
Finally, the Sobolev inequality follows from (\ref{213}) and (\ref{214}). We complete the proof of the lemma.
\end{proof}
Therefore, Theorem \ref{th1} follows directly from Lemma \ref{lm1} and Lemma \ref{lm2}.
\begin{proof}[Proof of Theorem \ref{th1}] Note that our log-Sobolev inequality in Lemma \ref{lm1} just has one more term $4tB_0A_0^{-1}$ than
Lemma \ref{lm2}. Applying the same arguments with the second part in the proof of Lemma \ref{lm2} to our log-Sobolev inequality, we have
\begin{align*}
\int_M f^{q}\md\mu(g(t))
\leq (2^{q}-1)2^{\frac{q(q-s)}{2-s}}e^{\frac{8tB_0}{A_0(n-2)}}\lf( 1+\frac{A_0 S_0}{4\lambda_1(t)}\rt)^{\frac{q}{2}}W(f)^q,
\end{align*}
for any $f\in W^{1,2}(M,g(t))$. Here
$$
W(f)=\lf[A_0\int_M\lf(|\nabla f|^2+\frac{1}{4}Sf^2\rt)\md\mu(g(t))+B_0\int_Mf^2\md\mu(g(t))\rt]^{\frac{1}{2}},
$$
$$\lambda_1(t)=\inf\lf\{W(f)^2|\int_M f^2\md \mu(g(t))=1, f\in W^{1,2}(M,g(t))\rt\},$$ and other symbols are all same as above.

By the definitions of $\mathcal{F}$ entropy and $W(f)$, we can see that $\lambda_1(t)$ is linearly related to the first eigenvalue of $\mathcal{F}$ entropy.
On the other hand, the first eigenvalue of $\mathcal{F}$ entropy is monotone non-decreasing in time (cf. Proposition 3.3 \cite{MR}).
So $\lambda_1(t)$ is also non-decreasing in time. It follows that our uniform Sobolev inequality holds, and we get
$$A(t)=(2^{q}-1)^{\frac{2}{q}}2^{\frac{2(q-s)}{2-s}}e^{\frac{8tB_0}{n A_0}}\lf( 1+\frac{A_0 S_0}{4\lambda_1(0)}\rt)A_0,$$
and $$B(t)=(2^{q}-1)^{\frac{2}{q}}2^{\frac{2(q-s)}{2-s}}e^{\frac{8tB_0}{n A_0}}\lf( 1+\frac{A_0 S_0}{4\lambda_1(0)}\rt)B_0,$$
where $S_0$ depends on $g_0$ and $\phi_0$, and $\lambda_1(0)$ depends on $A_0$, $B_0$, $g_0$ and $\phi_0$. Thus the proof of Theorem \ref{th1} is completed now.
\end{proof}
 From Theorem \ref{th1} we can obtain a uniform Sobolev
inequality along the harmonic-Ricci flow under the assumption of positive first eigenvalue of $\mathcal{F}$ entropy.
Recall that  $\lambda_0$ is the first eigenvalue of $\mathcal{F}$ entropy with respect to the initial metric $g_0$, i.e.
\begin{equation*}
\lambda_0=\inf_{\|v\|_2=1}\int_M(4|\nabla v|^2+Sv^2)\md\mu(g_0).
\end{equation*}
The $\mathcal{F}$ entropy corresponds to Perelman's $\mathcal{F}$ entropy for the Ricci flow introduced in \cite{P1}. Similarly as the Ricci flow, the
harmonic-Ricci flow can be interpreted as the gradient flow of $\mathcal{F}$ entropy modulo a pull-back by a family of diffeomorphisms.
The eigenvalue of $\mathcal{F}$ entropy is a very powerful tool for the research on Ricci flow and Riemannian manifolds. More results can be found in \cite{F3, LJ}.

Since $(M, g_0)$ is a closed Riemannian manifold of dimension $n\geq 3$, the Sobolev inequality holds, i.e. there exist positive constants $A$ and $B$  depending only on the initial metric $g_0$ such that, for any $v\in W^{1,2}(M)$,
\begin{equation}\label{215}
\lf(\int_M v^{\frac{2n}{n-2}}\md\mu(g_0)\rt)^{\frac{n-2}{n}}\leq A\int_M|\nabla v|^2\md\mu(g_0)+B\int_Mv^2\md\mu(g_0).
\end{equation}

If $\lambda_0>0$, combining with Sobolev inequality (\ref{215}), we have
\begin{eqnarray*}
\lf(\int_M v^{\frac{2n}{n-2}}\md\mu(g_0)\rt)^{\frac{n-2}{n}}\leq \tilde{A}_0 \int_M\lf(|\nabla v|^2+\frac{1}{4}Sv^2\rt)\md\mu(g_0)
\end{eqnarray*}
where $\tilde{A}_0$ depends only on initial metric $g_0$, $\phi_0$ and $\lambda_0$. This means that the assumption of Sobolev inequality in Theorem \ref{th1} at initial time holds with $B_0=0$.
Hence, Theorem \ref{th1} gives us the following result.
\begin{cor}\label{co1}
Let ($g(x,t), \phi(x,t)$) be a solution of the harmonic-Ricci flow (\ref{O1}) in $M^n\times[0,T)$ with initial value ($g_0, \phi_0$). Assume that the first eigenvalue $\lambda_0$ of $\mathcal{F}$ entropy with
respect to the initial metric $g_0$ is positive. Then there exists a positive constant $A$, depending only on $n$, $g_0$, $\phi_0$ and $\lambda_0$,
such that for all $v\in W^{1,2}(M,g(t))$, $t\in[0,T)$, it holds that
\begin{eqnarray}\label{216}
\lf(\int_M v^{\frac{2n}{n-2}}\md\mu\big(g(t)\big)\rt)^{\frac{n-2}{n}}\leq A \int_M \lf(|\nabla v|^2+\frac{1}{4}Sv^2\rt)\md\mu\big(g(t)\big).
\end{eqnarray}
\end{cor}
\begin{rem} The analogous uniform Sobolev inequalities were given in \cite{Z2} for the Ricci flow and \cite{LW} for the extended Ricci flow.
\end{rem}

\section{The proof of Theorem \ref{th2} and Theorem \ref{th3}}
\setcounter{equation}{0}
In this section, we first prove the linear parabolic estimate under the harmonic-Ricci flow with the help of the above Sobolev inequality (\ref{216}) and Moser's iteration. As a result, we derive an upper bound for the heat kernel, which is similar to the one known for the fixed metric case.

\begin{proof} [Proof of Theorem \ref{th2}] For any constant $p\geq 1$, it follows from (\ref{O3}) that
\begin{equation*}
\int_M f^p\partial_tf \md\mu-\int_Mf^p\Delta f \md\mu \leq a\int_Mf^{p+1}\md\mu,
\end{equation*}
where the volume measure $\md\mu=\md\mu(g(t))$ for simplicity, and the same symbol will also be used in the rest of the proof. Integrating by parts, we have
$$
\frac{1}{p+1}\int_M\partial_tf^{p+1}\md\mu+\frac{4p}{(p+1)^2}\int_M\left|\nabla f^{\frac{p+1}{2}}\right|^2\md\mu\leq a\int_Mf^{p+1}\md\mu.
$$
Since $\partial_t \md\mu=-S\md\mu$ and $4p\geq 2(p+1)$ for all $p\geq1$, multiplying both sides by $p+1$, we get
$$
\partial_t\int_Mf^{p+1}\md\mu+\int_MSf^{p+1}\md\mu+2\int_M\left|\nabla f^{\frac{p+1}{2}}\right|^2\md\mu\leq a(p+1)\int_Mf^{p+1}\md\mu.
$$
Notice that the condition $S\geq0$, then we have
\begin{equation}\label{31}
\partial_t\int_Mf^{p+1}\md\mu+\frac{1}{2}\int_M\left(Sf^{p+1}+4\left|\nabla f^{\frac{p+1}{2}}\right|^2\right)\md\mu\leq a(p+1)\int_Mf^{p+1}\md\mu.
\end{equation}
Next for any $0<\tau<\sigma<T$ we define
\begin{equation*}
\psi(t)=\left
\{\begin{array}{ll}
0&  0\leq t\leq\tau~,\\
(t-\tau)/(\sigma-\tau)&  \tau\leq t\leq \sigma~,\\
1&  \sigma\leq t\leq T~.
\end{array}\right.\
\end{equation*}
Multiplying (\ref{31}) by $\psi$, we obtain
\begin{equation*}
\partial_t\left(\psi\int_Mf^{p+1}\md\mu\right)+\frac{1}{2}\psi\int_M\left(Sf^{p+1}+4\left|\nabla f^{\frac{p+1}{2}}\right|^2\right)\md\mu
\leq [a(p+1)\psi+\psi']\int_Mf^{p+1}\md\mu.
\end{equation*}
Integrating this with respect to $t$ yields
\begin{align*}
\sup_{\sigma\leq t\leq T}\int_Mf^{p+1}\md\mu&+\frac{1}{2}\int_{\sigma}^T\int_M\left(Sf^{p+1}+4\left|\nabla f^{\frac{p+1}{2}}\right|^2\right)\md\mu \md t\\
&\leq 2\left[a(p+1)+\frac{1}{\sigma-\tau}\right]\int_{\tau}^T\int_Mf^{p+1}\md\mu \md t.
\end{align*}
By the assumption that $S\geq0$ at the initial time, the first eigenvalue $\lambda_0$ of $\mathcal{F}$ entropy with
respect to the initial metric $g_0$ is positive. Applying H\"{o}lder inequality, the above estimate and the Sobolev inequality (\ref{216}), we deduce
\begin{align*}
\int_{\sigma}^T\int_Mf^{(p+1)(1+\frac{2}{n})}\md\mu \md t & \leq \int_{\sigma}^T\left(\int_Mf^{p+1}\md\mu\right)^{\frac{2}{n}}\left(\int_Mf^{(p+1)\frac{n}{n-2}}\md\mu\right)^{\frac{n-2}{n}}\md t\\
&\leq \sup_{\sigma\leq t\leq T}\left(\int_Mf^{p+1}\md\mu\right)^{\frac{2}{n}}\int_{\sigma}^T
A\int_M\left(Sf^{p+1}+4\left|\nabla f^{\frac{p+1}{2}}\right|^2\right)\md\mu \md t\\
&\leq 4^{1+\frac{1}{n}}A\left[a(p+1)+\frac{1}{\sigma-\tau}\right]^{1+\frac{2}{n}}\left(\int_{\tau}^T\int_Mf^{p+1}\md\mu \md t\right)^{1+\frac{2}{n}}.
\end{align*}
Set
$$
H(p,\tau)=\left(\int_{\tau}^T\int_Mf^{p}\md\mu \md t\right)^{\frac{1}{p}},
$$
for any $p\geq 2$ and $0<\tau<T$. So we get
\begin{equation}\label{32}
H(p(1+\frac{2}{n}), \sigma)\leq \left(4^{1+\frac{1}{n}}A\right)^{\frac{1}{p(1+\frac{2}{n})}}\left(ap+\frac{1}{\sigma-\tau}\right)^{\frac{1}{p}}H(p,\tau).
\end{equation}
Now we fix $0<t_0<t_1<T$, $p_0\geq 2$. Let $\chi=1+\frac{2}{n}$, $p_k=p_0\chi^k$, $\tau_k=t_0+(1-\frac{1}{\chi^k})(t_1-t_0)$. Then it follows from (\ref{32}) that
\begin{equation*}
H(p_{k+1}, \tau_{k+1})\leq
\left(4^{1+\frac{1}{n}}A\right)^{\frac{1}{p_{k+1}}}\left(ap_k+\frac{\chi^k}{t_1-t_0}\frac{\chi}{\chi-1}\right)^{\frac{1}{p_k}}H(p_k,\tau_k).
\end{equation*}
Hence by iteration, we arrive at
\begin{equation*}
H(p_{m+1}, \tau_{m+1})\leq\left(4^{1+\frac{1}{n}}A\right)^{\sum\limits_{k=0}^m\frac{1}{p_{k+1}}}
\left(ap_0+\frac{1}{t_1-t_0}\frac{\chi}{\chi-1}\right)^{\sum\limits_{k=0}^m\frac{1}{p_k}}\chi^{\sum\limits_{k=0}^m\frac{k}{p_k}}H(p_0,\tau_0).
\end{equation*}
Letting $m\rightarrow\infty$, we obtain
\begin{equation*}
H(p_{\infty}, \tau_{\infty})\leq C_0
\left[ap_0+\frac{n+2}{2(t_1-t_0)}\right]^{\frac{n+2}{2p_0}}H(p_0,\tau_0).
\end{equation*} for all $p_0\geq 2$. This means
\begin{equation}\label{33}
\sup_{(x,t)\in M\times[t_1,T]}|f(x,t)|\leq
\left(C_1a+\frac{C_2}{t_1-t_0}\right)^{\frac{n+2}{2p_0}}\left(\int_{t_0}^T\int_Mf^{p_0}\md\mu \md t\right)^{\frac{1}{p_0}},
\end{equation}
where $C_1$ and $C_2$ are both positive constants depending on dimension $n$, $p_0$, initial metric $g_0$, $\phi_0$ and $\lambda_0$.
Moreover, for $0<p<2$, we set
$$
h(s)=\sup_{(x,t)\in M\times[s,T]}|f(x,t)|.
$$
Combining Young inequality with (\ref{33}), we deduce
\begin{align*}
h(t_1)&\leq \left(C_1a+\frac{C_2}{t_1-t_0}\right)^{\frac{n+2}{4}}\left(\int_{t_0}^T\int_Mf^{2}\md\mu \md t\right)^{\frac{1}{2}}\\
&\leq h(t_0)^{\frac{2-p}{2}}\left(C_1a+\frac{C_2}{t_1-t_0}\right)^{\frac{n+2}{4}}\left(\int_{t_0}^T\int_Mf^{p}\md\mu \md t\right)^{\frac{1}{2}}\\
&\leq \frac{1}{2}h(t_0)+\left(C_1a+\frac{C_2}{t_1-t_0}\right)^{\frac{n+2}{2p}}\left(\int_{t_0}^T\int_Mf^{p}\md\mu \md t\right)^{\frac{1}{p}}
\end{align*}
Now we use the iteration method again. Fix $0<t_0<t_1<T$, for some $0<\theta<1$, we let $x_k=t_1-(1-\theta^k)(t_1-t_0)$. Then by iteration
\begin{align*}
h(t_1)=h(x_0)&\leq \frac{1}{2^{k}}h(x_k)
+\left(\int_{t_0}^T\int_Mf^{p}\md\mu \md t\right)^{\frac{1}{p}}\sum_{i=0}^{k-1}\frac{1}{2^i}\left(C_1a+\frac{C_2}{x_i-x_{i+1}}\right)^{\frac{n+2}{2p}}\\
&\leq \frac{1}{2^{k}}h(x_k)
+\left(\int_{t_0}^T\int_Mf^{p}\md\mu \md t\right)^{\frac{1}{p}}\left(C_1a+\frac{C_2}{t_1-t_0}\right)^{\frac{n+2}{2p}}\sum_{i=0}^{k-1}\lf(2\theta^{\frac{n+2}{2p}}\rt)^{-i}
\end{align*}
Choose $0<\theta<1$ such that $2\theta^{\frac{n+2}{2p}}>1$, that is, $\frac{1}{2}<\theta^{\frac{n+2}{2p}}<1$. Taking $k\rightarrow\infty$, we have
\begin{align}\label{34}
h(t_1)\leq \left(C_1a+\frac{C_2}{t_1-t_0}\right)^{\frac{n+2}{2p}}\left(\int_{t_0}^T\int_Mf^{p}\md\mu \md t\right)^{\frac{1}{p}},
\end{align}
for all $0<p<2$ and $0<t_0<t_1<T$. By (\ref{33}) and (\ref{34}) together, as $t_0\rightarrow0$, it follows that
$$
h(t_1)\leq \left(C_1a+\frac{C_2}{t_1}\right)^{\frac{n+2}{2p}}\left(\int_{0}^T\int_Mf^{p}\md\mu \md t\right)^{\frac{1}{p}},\quad \forall p>0
$$
where $C_1$ and $C_2$ are both positive constants depending on dimension $n$, $p$, initial metric $g_0$, $\phi_0$ and $\lambda_0$. Thus we complete the proof now.
\end{proof}

Note that the heat equation (\ref{O2}) is a linear parabolic equation. The above proof for Theorem \ref{th2} can be applied to the heat equation almost verbatim. Thus it is not difficult to get the following corollary, which will be used to determine the upper bound of the heat kernel later.
\begin{cor}\label{co2}
Assume that ($g(x,t), \phi(x,t)$) is a smooth solution to the harmonic-Ricci flow (\ref{O1}) in $M^n\times[0,T]$ with the initial value
($g_0, \phi_0$) and $S\geq0$ at the initial time.
Let $u$ be a nonnegative smooth solution to the heat equation (\ref{O2}) on $M\times[0, T]$. Then for any $0\leq s<t\leq T$ and $p>0$ we have
\begin{equation*}
\sup_{x\in M} |u(x,t)|\leq \frac{C}{(t-s)^{\frac{n+2}{2p}}}\left(\int_{s}^t\int_Mu(x,\tau)^{p}\md\mu \md\tau\right)^{\frac{1}{p}},
\end{equation*}
where $C$ is a positive constant depending on dimension $n$, $p$, $g_0$, $\phi_0$ and $\lambda_0$.
\end{cor}
In fact, nonnegativity of $S$ in the conditions of Theorem \ref{th2} can be removed. At this time the parabolic estimate will also depend on the negative lower bound of $S$. By the similar arguments of Theorem \ref{th2} we can obtain the following estimate.
\begin{cor}\label{co3}
Assume that ($g(x,t), \phi(x,t)$) is a smooth solution to the harmonic-Ricci flow (\ref{O1}) in $M^n\times[0,T]$ with initial value ($g_0, \phi_0$),
$S\geq-S_0$ for a nonnegative constant $S_0$ at the initial time,
and the first eigenvalue $\lambda_0$ of $\mathcal{F}$ entropy with respect to $g_0$ is positive.
Let $f$ be a nonnegative Lipschitz continuous function on $M\times[0, T]$ satisfying
\begin{equation*}
\partial_t f\leq \Delta f+af
\end{equation*}
on $M\times[0, T]$ in the weak sense, where $a\geq 0$. Then we have for any $0<t\leq T$ and $p>0$
\begin{equation*}
\sup_{x\in M} |f(x,t)|\leq \left(C_0S_0+C_1a+\frac{C_2}{t}\right)^{\frac{n+2}{2p}}\left(\int_{0}^T\int_Mf^{p}\md\mu \md t\right)^{\frac{1}{p}},
\end{equation*}
where $C_0$, $C_1$ and $C_2$ are all positive constants depending on dimension $n$, $p$, $g_0$, $\phi_0$ and $\lambda_0$.
\end{cor}
Now we turn to control the heat kernel, which is easy to carry out by means of using the result in Corollary \ref{co2}.
\begin{proof} [Proof of Theorem \ref{th3}]
By the definition of heat kernel, we know the fact
$$
\partial_tG(x,t;y,s)-\Delta_xG(x,t;y,s)=0.
$$
Combining with the assumption of $S\geq0$, we have
$$
\partial_t\int_MG(x,t;y,s)\md\mu(x,t)=\int_M\big [\Delta_xG(x,t;y,s)-SG(x,t;y,s)\big ]\md\mu(x,t)\leq0.
$$
It implies that the above integral of heat kernel is non-increasing in $t$. So we derive
$$
\int_MG(x,t;y,s)\md\mu(x,t)\leq\int_MG(x,s;y,s)\md\mu(x,s)=1,
$$
for $\forall 0\leq s<t\leq T$. Therefore, by Corollary \ref{co2}, it follows that
$$
G(x,t;y,s)\leq \frac{C}{(t-s)^{\frac{n+2}{2}}}\int_{s}^t\int_MG(x,\tau;y,s)\md\mu(x,\tau)\md\tau\leq \frac{C}{(t-s)^{\frac{n}{2}}},
$$
where $C$ is a positive constant depending on dimension $n$, $g_0$, $\phi_0$ and $\lambda_0$.
\end{proof}
Thanks to Corollary \ref{co3}, the same upper bound of heat kernel can be given without the assumption of the nonnegativity of $S$ but at the cost of
the dependence on the upper bound for time.
\begin{cor}\label{co4}
Assume that ($g(x,t), \phi(x,t)$) is a smooth solution to the harmonic-Ricci flow (\ref{O1}) in $M^n\times[0,T]$ with initial value ($g_0, \phi_0$)
and the first eigenvalue $\lambda_0$ of $\mathcal{F}$ entropy with respect to $g_0$ is positive. Let $G(x,t;y,s)$ be the heat kernel.
Then for $\forall x,y\in M$ and $\forall 0\leq s<t\leq T$, we have
\begin{equation*}
G(x,t;y,s)\leq \frac{C}{(t-s)^{\frac{n}{2}}},
\end{equation*}
where $C$ is a positive constant, which depends on dimension $n$, $g_0$, $\phi_0$, $\lambda_0$ and $T$.
\end{cor}
\begin{proof} By Theorem \ref{th3}, we only need to prove the estimate for the case that $S$ has negative minimum at the initial time.
We can assume that $S_0=-\min\limits_{x\in(M,g_0)} S(x,0)>0$.
From the definition of heat kernel and Corollary \ref{co3}, we have
\begin{equation}\label{35}
\sup_{x\in M} |G(x,t;y,s)|\leq \lf(C_0S_0+\frac{C_2}{t-s}\rt)^{\frac{n+2}{2}}\int_{s}^t\int_MG(x,\tau;y,s)\md\mu(x,\tau) \md\tau,
\end{equation}
for $\forall 0\leq s<t\leq T$.
In addition, since $S>-S_0$ for any time $t\in [0,T]$, we deduce
\begin{align*}
\partial_t\int_MG(x,t;y,s)\md\mu(x,t)=&\int_M\big [\Delta_xG(x,t;y,s)-SG(x,t;y,s)\big ]\md\mu(x,t)\\
\leq& S_0\int_MG(x,t;y,s)\md\mu(x,t).
\end{align*}
Integrating from $s$ to $\tau$ gives
$$
\int_MG(x,\tau;y,s)\md\mu(x,\tau)\leq e^{S_0(\tau-s)}\int_MG(x,s;y,s)\md\mu(x,s)=e^{S_0(\tau-s)},
$$
for $\forall 0\leq s<\tau\leq T$. Therefore, combining with (\ref{35}), we conclude that
$$
G(x,t;y,s)\leq \lf(C_0S_0+\frac{C_2}{t-s}\rt)^{\frac{n+2}{2}}\int_{s}^te^{S_0(\tau-s)}\md\tau\leq \frac{C}{(t-s)^{\frac{n}{2}}},
$$
where the last inequality follows from $0<t-s\leq T$ and $C$ is a positive constant depending on dimension $n$, $g_0$, $\phi_0$, $\lambda_0$ and $T$.
\end{proof}
\noindent{\bf{Acknowledgements.}} This work was carried out while the authors were visiting Northwestern Univerisity. We would like to thank Professor Valentino Tosatti and Professor Ben Weinkove for hospitality and helpful discussions. We also thank Wenshuai Jiang for some useful conversations.

\begin{flushleft}
Shouwen Fang\\
School of Mathematical Science, Yangzhou University,\\
Yangzhou 225002, P. R. China\\
E-mail: shwfang@163.com
\end{flushleft}

\begin{flushleft}
Tao Zheng\\
School of Mathematics and Statistics, Beijing Institute of Technology,\\
 Beijing 100081, P. R. China\\
E-mail: zhengtao08@amss.ac.cn
\end{flushleft}


\begin{thebibliography}{50}
\bibitem{M}{\sc M. Bailesteanu}, Bounds on the heat kernel under the Ricci flow, \textit{Proc. Amer. Math. Soc.} \textbf{140} (2012), 691--700.

\bibitem{MH}{\sc M. Bailesteanu and H. Tran}, Heat kernel estimates under the Ricci-harmonic map flow, \textit{arXiv:math.DG/1310.1619}.

\bibitem{BCLS}
{\sc D. Bakry, T. Coulhon, M. Ledoux and L. Saloff-Coste}, Sobolev inequalities in disguise, \textit{Indiana Univ. Math. J.} \textbf{44} (1995),
 1033--1047.

\bibitem{CZ}{\sc H. D. Cao and X. P. Zhu}, A complete proof of the Poincar\'{e} and Geometrization conjectures---Application of the
Hamilton-Perelman theory of the Ricci flow, \textit{Asian J. Math.} \textbf{10} (2006), 165--492.

\bibitem{CLN}{\sc B. Chow, P. Lu and L. Ni}, Hamilton¡¯s Ricci flow, \textit{Graduate Studies in Mathematics}, vol. 77, American
              Mathematical Society, Providence, RI, 2006.

\bibitem{F1}{\sc S. W. Fang}, Differential Harnack inequalities for heat equations with potentials under the Bernhard List's flow, \textit{Geom. Dedicata} \textbf{161} (2012), 11--22.

\bibitem{F2}{\sc S. W. Fang}, Differential Harnack inequalities for backward heat equations with potentials under an extended Ricci flow, \textit{Adv. Geom.} \textbf{13} (2013), 741--755.

\bibitem{F3}{\sc S. W. Fang, H. F. Xu and P. Zhu}, Evolution and monotonicity of eigenvalues under the Ricci flow, \textit{Sci. China Math.} \textbf{58} (2015), doi: 10.1007/s11425-014-4943-7.

\bibitem{F4}{\sc S. W. Fang and P. Zhu}, Differential Harnack estimates for backward heat equations with potentials under geometric flows, \textit{Commun. Pur. Appl. Anal.}, to appear.

\bibitem{H1}{\sc R. S. Hamilton}, Three-mainfolds with positive Ricci curvature, \textit{J. Differ. Geom.} \textbf{17} (1982), 255--306.

\bibitem{H2}{\sc R. S. Hamilton}, Four manifolds with positive Ricci curvature, \textit{J. Differ. Geom.} \textbf{24} (1986), 153--179.

\bibitem{Hs}{\sc S. Y. Hsu}, Uniform Sobolev inequalities for manifolds evolving by Ricci flow, \textit{arXiv:math.DG/ 0708.0893}.

\bibitem{J}{\sc W. S. Jiang}, Bergman kernel along the K\"ahler-Ricci flow and Tian's conjecture, \textit{J. reine angew. Math.} (2014), doi: 10.1515/crelle-2014-0015.

\bibitem{KZ}{\sc S. L. Kuang and Q. S. Zhang}, A gradient estimate for all positive solutions of the conjugate heat equation under Ricci flow, \textit{J. Funct. Anal.} \textbf{255} (2008), 1008--1023.

\bibitem{LJ}{\sc J. F. Li}, Eigenvalues and energy functionals with monotonicity formulae under Ricci flow, \textit{Math. Ann.} \textbf{338} (2007), 927--946.
\bibitem{LY}{\sc P. Li and S. T. Yau}, On the parabolic kernel of the Schr\"{o}dinger operator, \textit{Acta Math.} \textbf{156} (1986), 153--201.

\bibitem{BL}{\sc B. List}, Evolution of an extended Ricci flow system,  \textit{Comm. Anal. Geom.} \textbf{16} (2008), 1007--1048.

\bibitem{LW}{\sc X. G. Liu and K. Wang}, A Gaussian upper bound of the conjugate heat equation along an extended Ricci flow, \textit{arXiv:math.DG/1412.3200}.

\bibitem{MR}{\sc R. M\"{u}ller}, Ricci flow coupled with harmonic map flow, \textit{Ann. Sci. Ecole Norm. S.} \textbf{45} (2012), 101--142.

\bibitem{P1}{\sc G. Perelman}, The entropy formula for the Ricci flow and its geometric applications,
              \textit{arXiv:math.DG/ 0211159}.

\bibitem{W}{\sc J. Wang}, Global heat kernel estimates, \textit{Pacific J. Math.} \textbf{178} (1997), 377--398.

\bibitem{Y1}{\sc R. G. Ye}, Curvature estimates for the Ricci flow I, \textit{Calc. Var.} \textbf{31} (2008), 417--437.

\bibitem{Y2}{\sc R. G. Ye}, The logarithmic Sobolev inequality along the Ricci flow, \textit{arXiv:math.DG/0707.2424}.

\bibitem{Z1}{\sc Q. S. Zhang}, Some gradient estimates for the heat equation on domain for an equation by
    Perelman, \textit{International Mathematics Research Notices} \textbf{2006} (2006), article id: 92314, 1--39.

\bibitem{Z2}{\sc Q. S. Zhang}, A uniform Sobolev inequality under Ricci flow, \textit{International Mathematics Research Notices} \textbf{2007} (2007), article id: rnm056, 1--17.

\bibitem{Z3}{\sc Q. S. Zhang}, Sobolev inequalities, heat kernels under Ricci flow and Poincar\'{e} conjecture, CRC Press,
                   Boca Raton, FL, 2010.

\bibitem{Zh}{\sc A. Q. Zhu},Differential Harnack inequalities for the backward heat equation with potential
          under the harmonic-Ricci flow, \textit{J. Math. Anal. Appl.} \textbf{406} (2013), 502--510.


\end{thebibliography}
\end{document}